\documentclass[12pt]{article}
\usepackage{amsthm,amsmath,amssymb,latexsym,amscd}
\usepackage{graphicx}
\newcommand{\g}{\frak{g}}

\newcommand{\D}{\mathcal{D}}
\newcommand{\F}{\mathcal{F}}
\newcommand{\I}{\mathcal{I}}
\renewcommand{\O}{\mathcal{O}}
\renewcommand{\P}{\mathcal{P}}
\newcommand{\Q}{\mathcal{Q}}
\renewcommand{\S}{\mathcal{S}}
\newcommand{\T}{\mathcal{T}}
\newcommand{\Der}{\mathrm{Der}}

\renewcommand{\H}{\mathrm{H}}
\newcommand{\Coder}{\mathrm{Coder}}
\newcommand{\card}{\mathrm{card}}
\newcommand{\sgn}{\mathrm{sgn}}
\newcommand{\Ass}{\mathcal{A}ss}
\newcommand{\Com}{\mathcal{C}om}
\newcommand{\Lie}{\mathcal{L}ie}
\newcommand{\Leib}{\mathcal{L}eib}

\newcommand{\Perm}{\mathcal{P}erm}

\newcommand{\Zinb}{\mathcal{Z}inb}

\renewcommand{\d}{\mathbf{d}}

\renewcommand{\c}{\circ}
\newcommand{\ol}{\overline}
\newcommand{\ot}{\otimes}
\newcommand{\s}{\mathbf{s}}

\newtheorem{definition}{Definition}[section]
\newtheorem{lemma}[definition]{Lemma}

\newtheorem{proposition}[definition]{Proposition}
\newtheorem{theorem}[definition]{Theorem}
\newtheorem{corollary}[definition]{Corollary}
\newtheorem{remark}[definition]{Remark}
\newtheorem{example}[definition]{Example}
\date{}

\begin{document}

\title{
Derived bracket construction up to homotopy
and Schr\"{o}der numbers
}
\author{K. UCHINO}
\maketitle
\abstract
{
On introduit la notion de la construction crochet d\'eriv\'e sup\'erieure
dans la cat\'egorie des op\'erades.
On prouve que la construction crochet d\'eriv\'e sup\'erieure
de l'op\'erade $\Lie$ est identique \'{a} la construction cobar
de l'op\'erade $\Leib$ de Jean-Louis Loday.
Ce th\'eor\`em est d\'emontr\'e par le calcul du nombre de Ernst Schr\"oder.
On trouve que la collection d'arbres racin\'es \'etiquet\'es
peut \^etre d\'ecompos\'e par l'op\'erade $\Lie$ et une nouvelle op\'erade.
%We introduce an operadic version of higher derived bracket construction.
%It has been proved that the higher derived bracket construction
%of the operad of Lie algebras coincides to cobar construction
%of Leibniz operad.
}

\section{Introduction}
The aim of this note is to prove an identity below.
\medskip\\
\noindent
\textbf{Theorem}.
$\s\Leib_{\infty}\cong\Lie\ot\D_{\infty}$,
\medskip\\
\noindent
where
$\s(-)$ is an operadic suspension,
$\Leib$ is the operad of Leibniz (or Loday) algebras,
$\Lie$ is the one of Lie algebras,
$\Leib_{\infty}$ is the strong homotopy version of $\Leib$ and
$\D_{\infty}=(\D_{\infty},\d)$ is a new dg operad,
which is called a \textbf{deformation operad}.
Here $\d$ is a differential on $\D_{\infty}$
and the tree differential on $\s\Leib_{\infty}$ is equivalent to $\Lie\ot\d$.
Remark that $\D_{\infty}$ is not strong homotopy operad in usual sense.
The following identity is already known,
$$
\s\Leib\cong\Lie\ot\s\Perm,
$$
where $\Perm$ is Chapoton's permutation operad \cite{Chap}.
Theorem is regarded as a homotopy version of this classical identity.\\
\indent
The dg operad $\D_{\infty}$ is defined as a deformation of $\s\Perm$.
The operad $\s\Perm$ can be constructed
with a formal differential $d_{0}$ and the commutative associative operad $\Com$.
The deformation operad will be constructed by using a deformation differential,
$\hbar d_{1}+\hbar^{2}d_{2}+\cdots$, instead of $d_{0}$.
As a corollary of Theorem we prove that
$(\D_{\infty},\d)$ is a resolution over $\s\Perm$.\\
\indent
To prove Theorem we compute (small-)\textbf{Schr\"{o}der numbers}\footnote{
Schr\"{o}der numbers are sometimes called super Catalan numbers.
}
(See Table 1.)
\begin{table}[h]
\begin{center}
\begin{tabular}{c||c|c|c|c|c|c|c|c|c|c}
\hline
$n$&1&2&3&4 &5 &6  &7 &8  &9  &10  \\ \hline
$s(n)$&1&1&3&11&45&197&903&4279&20793&103049    \\
\hline
\end{tabular}
\caption{Schr\"{o}der numbers}
\end{center}
\end{table}
The Schr\"{o}der number $s(n)$ is known as the cardinal number of planar rooted trees with $n$-leaves,
on the other hand, $\s\Leib_{\infty}$ is, as an operad,
isomorphic to the operad of {\em labeled} planar rooted trees.
Theorem says that the set of labeled planar rooted trees is decomposed into
$\Lie$ and $\D_{\infty}$. This gives a new interpritation of Schr\"{o}der numbers.
\medskip\\
\indent
The generator of $\Lie\ot\D_{\infty}$ has the following form,
$$
\{\{\{\{d_{n-1}(1),2\},3\},...,n\}.
$$
Here $\{,\}$ is a Lie bracket in $\Lie(2)$ and $d_{n}$ is a formal derivation.
The bracket of this type is called a \textbf{higher derived bracket}
or derived bracket up to homotopy.
In particular, when $n=2$, $\{d_{1}(1),2\}$ is called a \textbf{binary derived bracket}.
The notion of (binary) derived bracket was defined by Kosmann-Schwarzbach \cite{Kos1}
in the study of Poisson geometry (see also \cite{Kos2}.)
The higher version was introduced by several authors
(cf. Roytenberg \cite{Roy}, Voronov \cite{Vo}.)
%Usefulness of deraived bracket has been demonstrated by many authors.
Derived brackets were born in Poisson geometry,
however, an important development of {\em derived bracket theory}
was made in the study of algebraic operads by Aguiar \cite{Ag}.
Aguiar discovered that several types of algebras are induced by
the method of derived bracket construction.
Inspired by Aguiar's work, Uchino \cite{U1} introduced the notion of
binary derived bracket construction {\em on the level of operad}.
This is an endofunctor on the category of binary quadratic operads
defined by applying the permutation operad,
$(-)\ot\s\Perm:\P\mapsto\P\ot\s\Perm$.
In non graded case, $(-)\ot\Perm$.
The derived bracket {\em on the level of algebra}
is regarded as a representation of this functor.\\
\indent
In this article, we will try to make a higher version of $(-)\ot\s\Perm$.
Our solution is not $(-)\ot\s\Perm_{\infty}$, but the functor $(-)\ot\D_{\infty}$
appeared in Theorem above.
We call this functor a {\em higher derived bracket construction} (on the level of operad.)
The strong homotopy Leibniz operad $\s\Leib_{\infty}$ is
the result of {\em cobar construction} with Koszul duality theory (Ginzburg and Kapranov \cite{GK}.)
Hence the theorem means
that the higher derived bracket construction of $\Lie$ coincides
with the cobar construction of $\Leib$.
An advantage of the higher derived bracket construction
is that it uses no Koszul duality theory.
%\medskip\\
%\noindent
%\textbf{Acknowledgement}....

%%%%%%%%%%%%%%%%%%%%%%%%%%%%
\section{Preliminaries}
%%%%%%%%%%%%%%%%%%%%%%%%%%%%

\subsection{Assumptions and Notations}

Through the paper, all algebraic objects are assumed to be defined
over a fixed field $\mathbb{K}$ of characteristic zero.
The mathematics of graded linear algebra is due to Koszul sign convention.
Namely, the transposition of tensor product satisfies
$o_{1}\ot o_{2}\cong(-1)^{|o_{1}||o_{2}|}o_{2}\ot o_{1}$
for any objects $o_{1}$ and $o_{2}$,
where $|o_{i}|$ is the degree of $o_{i}$.
We denote by $s(-)$ a suspension of degree $+1$.
For any object $o$, the degree of $s(o)$ is $|s(o)|=|o|+1$.
The inverse of $s$ is $s^{-1}$, whose degree is $-1$.

\subsection{Algebraic operads}

We refer the readers to Loday \cite{Lod1,Lod3,Lod2} and
Loday-Vallette \cite{LV}, for the details of algebraic operad theory.\\
\indent
An $\S$-module, $\P:=(\P(1),\P(2),...)$, is by definition a collection of $S_{n}$-modules $\P(n)$,
where $S_{n}$ is the $n$th symmetric group.
The notion of morphism between $\S$-modules
is defined by the usual manner, i.e.,
it is a collection $\phi=(\phi(1),\phi(2),...)$ of equivariant linear mappings $\phi(n):\P_{1}(n)\to\P_{2}(n)$.
Here $\P_{1}$ and $\P_{2}$ are any $\S$-modules.
Thus the category of $\S$-modules is defined.\\
\indent
In the category of $\S$-modules,
a tensor product, $\P_{1}\odot\P_{2}$, is defined by
$$
(\P_{1}\odot\P_{2})(n):=\bigoplus_{m,l}
\P_{1}(m)\ot_{S_{m}}\Big(\P_{2}(l_{1}),...,\P_{2}(l_{m})\Big)\ot_{(S_{l_{1}},...,S_{l_{m}})}S_{n},
$$
where $n=l_{1}+\cdots+l_{m}$ and $l:=(l_{1},...,l_{m})$.
It is easy to see that the tensor product is associative.
We consider a special $\S$-module $\I:=(\mathbb{K},0,0,...)$.
It is easy to check that
$\I\odot\P\cong\P\cong\P\odot\I$.
The concept of binary product on $\P$ is defined as a morphism of $\gamma:\P\odot\P\to\P$.
\begin{definition}[algebraic operad]
A triple $\P:=(\P,\I,\gamma)$ is called an algebraic operad, or shortly operad,
if it is a unital monoid in the category of $\S$-modules.
\end{definition}
If $\P$ is an operad,
$\P(n)$ is considered to be a space of formal $n$-ary operations.
For example $\P(2)$ is a space of formal binary operations,
which are usually denoted by $1*2$, $1\cdot 2$, $[1,2]$, $\{1,2\}$ and so on.
The operad structure $\gamma$ defines a composition product
on the operations, for instance,
\begin{eqnarray*}
&&\gamma:\{1,2\}\ot([1,2], 1*2)\mapsto\{[1,2],3*4\},\\
&&\gamma:[2,1]\ot(1_{\P},\{1,2\})\mapsto[\{2,3\},1],
\end{eqnarray*}
where $1_{\P}$ is the unite element of $\P$.
The numbers $1,2,3,...$ in the formal products are called \textbf{labels} or \textbf{leaves}.
\begin{definition}
For any  $p_{m}\in\P(m)$, $p_{n}\in\P(n)$,
$$
p_{m}\c_{i}p_{n}:=\gamma\Big(p_{m}\ot\big(1_{\P}^{\ot(i-1)} \ , \ p_{n} \ , \ 1_{\P}^{\ot(m-i)}\big)\Big).
$$
where $1\le i\le m$.
\end{definition}
The composition $p_{m}\c_{i}p_{n}$ values in $\P(m+n-1)$.
The structure $\gamma$ is decomposed into the compositions $(\c_{1},\c_{2},...)$.
\begin{definition}[free operad]
Let $\P$ be an $\S$-module not necessarily operad.
The free operad over $\P$, which is denoted by $\T\P$,
is by definition the free unital monoid in the category of
$\S$-modules.
\end{definition}
It is easy to see that $(\T\P)(1)\cong T\P(1)$ the tensor algebra over $\P(1)$.
We denote the quadratic part of $\T\P$ by $\T^{2}\P$,
$$
(\T^{2}\P)(m+n-1)=<p_{m}\c_{i}p_{n} \ | \ 1\le{i}\le m>,
$$
in particular, $(\T^{2}\P)(1)=\P(1)\c_{1}\P(1)\cong\P(1)\ot\P(1)$.
\begin{definition}[quadratic operad]
Let $R\subset\T^{2}\P$ be a sub $\S$-module of $\T\P$.
The quotient operad $\O:=\T^{2}\P/(R)$ is called
a quadratic operad, where $(R)$ is an ideal generated by $R$.
The generator $R$ is called a quadratic relation.
If $\P=\P(2)$ with $\P(n\neq 2)=0$,
$\O$ is called a binary quadratic operad.
\end{definition}
We recall two examples of binary quadratic operads.
The Lie operad, $\Lie$, is a binary quadratic operad generated by
a formal skewsymmetry bracket $\{1,2\}(=-\{2,1\})$.
$$
\Lie:=\T(\{1,2\})/(R_{\Lie}).
$$
The quadratic relation $R_{\Lie}$ is generated by the Jacobi identity,
$$
\{\{1,2\},3\}+\{\{3,1\},2\}+\{\{2,3\},1\}=0.
$$
We obtain the following expression of $\Lie$.
\begin{eqnarray*}
\Lie(1)&=&\mathbb{K},\\
\Lie(2)&=&<\{1,2\}>,\\
\Lie(3)&=&<\{\{1,2\},3\} \ , \ \{\{1,3\},2\}>,\\
\cdots&\cdots&\cdots.
\end{eqnarray*}
\begin{lemma}
$\dim\Lie(n)=(n-1)!$.
\end{lemma}
\begin{proof}
An arbitrary bracket in $\Lie(n)$ is generated by
the right-normed brackets
$$
\{\sigma(1),\{\sigma(2),...,\{\sigma(n-1),n\}\}\},
$$
where $\sigma\in S_{n-1}$.
\end{proof}
The commutative associative operad, $\Com$, is
generated by a formal commutative product.
We denote by $1\ot 1$ the commutative product.
$$
\Com:=\T(1\ot 1)/(R_{\Com}).
$$
The quadratic relation $R_{\Com}$ is the associative law,
$$
(1\ot 1)\ot 1-1\ot(1\ot 1)=0.
$$
We obtain the following expression of $\Com$.
\begin{eqnarray*}
\Com(1)&=&\mathbb{K},\\
\Com(2)&=&<1\ot 1>,\\
\Com(3)&=&<1\ot 1\ot 1>,\\
\cdots&\cdots&\cdots.
\end{eqnarray*}
It is obvious that $\Com(n)\cong\mathbb{K}$ for each $n$.\\
\indent
In the final of this subsection, we recall some basic concepts in algebraic operad theory.
\begin{definition}[operadic suspension]\label{jyunbi1}
If $\P=(\P(n))$ is an operad, the shifted operad $\s\P$ is defined by
$$
(\s\P)(n)\cong s^{-1}\ot\P(n)\ot s^{\ot n}\cong s^{n-1}\P(n)\ot\sgn_{n},
$$
where $\sgn_{n}$ is the sign representation of $S_{n}$.
The inverse of $\s$, $\s^{-1}$, is defined by the same manner.
\end{definition}
\begin{definition}[dg operad]\label{jyunbi2}
By definition, a differential graded operad, or shortly dg operad,
is an operad such that for each $n$
$(\P(n),d)$ is a complex and the differential is compatible with
the operad structure, i.e., it is equivariant and satisfies
the usual condition,
$$
d(p_{m}\c_{i}p_{n})=dp_{m}\c_{i}p_{n}+(-1)^{|p_{m}|}p_{m}\c_{i}dp_{n},
$$
where $p_{m}\in\P(m)$ and $p_{n}\in\P(n)$.
\end{definition}
\begin{definition}[Koszul dual operad]\label{jyunbi3}
Let $\P=\T(E)/(R)$ be a binary quadratic operad with $\P(2)=E$.
We put $E^{\vee}:=E^{*}\ot\sgn_{2}$.
The Koszul dual of $\P$ is by definition
$$
\P^{!}:=\T(E^{\vee})/(R^{\bot}),
$$
where $R^{\bot}$ is the orthogonal space of $R$.
\end{definition}
It is obvious that $\P^{!!}\cong\P$.
It is well-known that $\Lie^{!}\cong\Com$.

\subsection{(Sh) Leibniz operad}

We recall the notion of (sh) Leibniz algebras.
\begin{definition}[Leibniz/Loday algebras \cite{Lod1,Lod2}]
A Leibniz algebra or Loday algebra $(L,[,])$ is by definition
a vector space equipped with a binary bracket product $[,]$
satisfying the Leibniz identity,
$$
[x_{1},[x_{2},x_{3}]]=[[x_{1},x_{2}],x_{3}]+[x_{2},[x_{1},x_{3}]],
$$
where $x_{\cdot}\in L$.
\end{definition}
The operad of Leibniz algebras is denoted by $\Leib$, which is a binary quadratic operad
generated by $[1,2]$ and $[2,1]$,
$$
\Leib:=\T([1,2],[2,1])/(R_{\Leib}),
$$
where $R_{\Leib}$ is the Leibniz identity.
If the degree of $[1,2]$ is odd, i.e., $[1,2]\in\s\Leib(2)$, then
the Leibniz identity has the following form,
$$
[1,[2,3]]=-[[1,2],3]-[2,[1,3]],
$$
which is called an odd Leibniz identity.\\
\indent
We recall sh Leibniz algebras (cf. Ammar and Poncin \cite{AP}) and its operad.
\begin{definition}[Koszul dual of Leibniz algebra \cite{Lod2}, Zinbiel \cite{Zinb}]
A Zinbiel algebra $(Z,*)$ is by definition
a vector space equipped with a binary product $*$
satisfying
$$
x_{1}*(x_{2}*x_{3})=(x_{1}*x_{2})*x_{3}+(x_{2}*x_{1})*x_{3},
$$
which is called a Zinbiel identity or dual Leibniz identity.
\end{definition}
The operad of Zinbiel algebras is denoted by $\Zinb$,
which is the Koszul dual of $\Leib$, that is, $\Zinb=\Leib^{!}$.
It is known that $\Zinb(n)\cong S_{n}$ for each $n$ (\cite{Lod2}).\\
\begin{lemma}[\cite{AP}]
The cofree\footnote{in the category of nilpotent coalgebras.}
Zinbiel coalgebra over a space $V$ is the tensor space
$$
\bar{T}^{c}V:=V\oplus V^{\ot 2}\oplus\cdots
$$
equipped with the coproduct defined by
$$
\Delta(x_{1},...,x_{n+1}):=\sum_{i=1}^{n}\sum_{\sigma}\epsilon(\sigma)
(x_{\sigma(1)},...,x_{\sigma(i)})\ot(x_{\sigma(i+1)},...,x_{\sigma(n)},x_{n+1}),
$$
where $\epsilon(\sigma)$ is a Koszul sign and
$\sigma$s are $(i,n-i)$-unshuffles permutations, that is,
$\sigma(1)<\cdots<\sigma(i)$ and $\sigma(i+1)<\cdots<\sigma(n)$.
\end{lemma}
Let $\Coder(\bar{T}^{c}V)$ be the space of coderivations on the coalgebra.
\begin{definition}[\cite{AP}]
Let $\partial\in\Coder(\bar{T}^{c}V)$ be a coderivation of degree $+1$ on the coalgebra.
The pair $(V,\partial)$ is called an sh Leibniz algebra
or sh Loday algebra, if $\partial^{2}=0$, that is, codifferential.
\end{definition}
In general, the structure of sh Leibniz algebra $\partial$ has the form of deformation,
$$
\partial=\partial_{1}+\partial_{2}+\cdots.
$$
For each $j$, the coderivation $\partial_{j}$ is on $V^{\ot j}$ identified to
a linear map of $\partial_{j}:V^{\ot j}\to V$, and on $V^{\ot n\ge j}$ it satisfies
$$
\partial_{j}(x_{1},...,x_{n})=
\sum_{k=j}^{n}\sum_{\sigma}(\pm)(x_{\sigma(1)},...,x_{\sigma(k-j)},
\partial_{j}(x_{\sigma(k-j+1)},...,x_{\sigma(k-1)},x_{k}),x_{k+1},...,x_{n}),
$$
where $(\pm)$ is an appropriate sign and $\sigma$ are $(k-j,j-1)$-unshuffle permutations.
The defining condition of sh Leibniz algebras, $\partial^{2}=0$, is equivalent to
$$
[\partial_{1},\partial_{n}]+\sum_{i+j-1=n}\partial_{i}\partial_{j}=0.
$$
\indent
There is an easy method of making sh Leibniz algebras
(so-called higher derived bracket construction {\em on the level of algebra}.)
\begin{proposition}[\cite{U2}]\label{proplemma}
Let $(\g,\{,\},d_{0})$ be a dg Lie algebra
with a differential $d_{0}$. There exists a Lie algebra homomorphism
$$
N:\Der(\g)[[\hbar]]\to\Coder(\bar{T}^{c}\g),
$$
where $\Der(\g)$ is the space of derivations on $\g$.
Suppose that $d_{\hbar}:=d_{0}+\hbar d_{1}+\hbar^{2}d_{2}+\cdots$
is a deformation differential of $d_{0}$.
We define $\partial_{n+1}:=N(\hbar^{n}d_{n})$ for each $n$.
Then $\partial=\partial_{1}+\partial_{2}+\cdots$
becomes an sh Leibniz algebra structure.
\end{proposition}
\begin{proof}(Sketch)
The map of the proposition is defined as the higher derived bracket,
$$
N(\hbar^{n}D)(x_{1},x_{2},...,x_{n+1}):=\{\{\{Dx_{1},x_{2}\},...\},x_{n+1}\},
$$
for any $\hbar^{n}D\in\Der(\g)[[\hbar]]$.
\end{proof}
But now to our next task.
We consider an $\S$-module $s\ol{\Zinb^{*}}:=(s\Zinb^{*}(n))$ with $n\ge 2$,
where $\Zinb^{*}(n)$ is the dual space of $\Zinb(n)$.
We denote by $T_{n}:=T_{n}(1,2,...,n)$ the generator of $s\Zinb^{*}(n)$.
Let $\T(s\ol{\Zinb^{*}})$ be the free operad over the $\S$-module.
This operad is generated from $(T_{n})$.
One can define a differential, $d_{t}$, on the free operad by
\begin{eqnarray}
\label{TT01}d_{t}T_{2}&:=&0,\\
\label{TT02}d_{t}T_{n}+\sum_{i+j-1=n}T_{i}T_{j}&:=&0,
\end{eqnarray}
where
\begin{equation}\label{TT03}
T_{i}T_{j}:=\sum_{k=j}^{n}\sum_{\sigma}(T_{i}\c_{k-j+1}T_{j})
(\sigma(1),...,\sigma(k-j),\sigma(k-j+1),...,\sigma(k-1),k,k+1,...,n),
\end{equation}
where $k$, $\sigma$ are the same as above.
It is easy to check $d_{t}d_{t}=0$.
This differential is called a \textbf{tree differential}.
\begin{definition}[sh Leibniz operad]\label{defshleib}
$\s\Leib_{\infty}:=\Big(\T(s\ol{\Zinb^{*}}),d_{t}\Big)$.
\end{definition}
Hence $\Leib_{\infty}=\s^{-1}\T(s\ol{\Zinb^{*}})$.
\medskip\\
\indent
In the final of this section, we recall the concept of tree.
A planar rooted tree with $n$-leaves
is by definition a directed graph with
$n$-leaves (input edges),
$1$-root (output edge)
and without loop (See Fig 1).
A labeled planar rooted tree is a planar rooted tree
whose leaves are labeled by natural numbers.
\begin{figure}[h]
\begin{center}
\includegraphics[width=80mm]{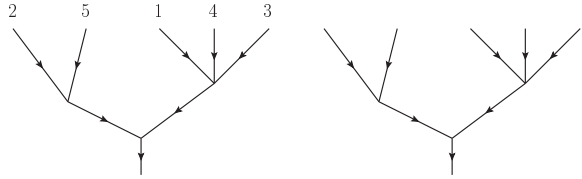}
\end{center}
\caption{labeled $5$-tree and non-labeled one}
\end{figure}
\begin{lemma}\label{firstlemma}
As an operad, up to degree, $\s\Leib_{\infty}$ is isomorphic to
the labeled planar rooted trees, i.e.,
$\s\Leib_{\infty}(n)$ is linearly isomorphic to
the space of labeled planar rooted trees with $n$-leaves.
\end{lemma}
\begin{proof}
Since $\Zinb(n)\cong S_{n}$,
the generator $T_{n}\in s\Zinb^{*}(n)$ is identified to
a labeled planar rooted tree with $n$-leaves and with $1$-internal vertex (See Fig 2.)
\begin{figure}[h]
\begin{center}
\includegraphics[width=40mm]{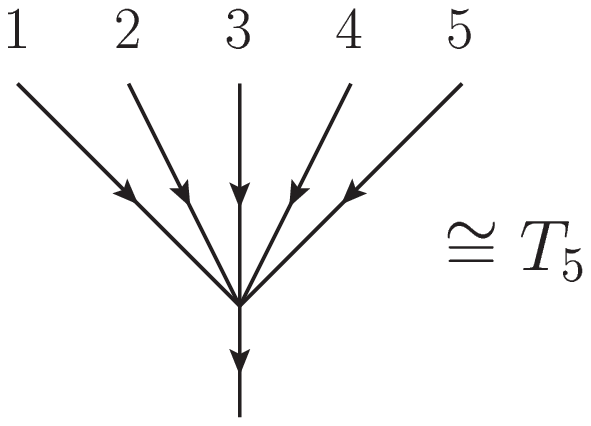}
\end{center}
\caption{$T_{5}$}
\end{figure}
\end{proof}
The lemmas above will be used in Section 4.

%%%%%%%%%%%%%%%%%%%%%%%%%%%
\section{Deformation operad}
%%%%%%%%%%%%%%%%%%%%%%%%%%%%

%%%%%%%%%%%%%%%%%%%%%%%%%%%%%%%%%%%%%%%%%%%%%%%%
\subsection{Permutation algebras and the operad $\Perm$}
%%%%%%%%%%%%%%%%%%%%%%%%%%%%%%%%%%%%%%%%%%%%%%%%

A permutation algebra introduced by Chapoton \cite{Chap}
is by definition an associative algebra $(A,*)$ satisfying
$$
(a_{1}*a_{2})*a_{3}=a_{1}*(a_{2}*a_{3})=(a_{2}*a_{1})*a_{3}.
$$
The operad of permutation algebras is denoted by $\Perm$,
which is also a binary quadratic operad
$$
\Perm:=\T(1*2,2*1)/(R_{\Perm}).
$$
\indent
We recall a construction of $\Perm$ with a formal differential (\cite{U1}).
Let $d_{0}$ be a 1-ary operator of degree $+1$
and let $1\ot 1\in\Com(2)$ be a binary commutative product of degree $0$.
Let $\T(d_{0},1\ot 1)$ be the free operad over $(d_{0},1\ot 1)$.
We define a quadratic operad $\Q$,
$$
\Q:=\T(d_{0},1\ot 1)/(R_{\Q}),
$$
where $R_{\Q}$ is the space of three quadratic relations,
\begin{eqnarray*}
(1\ot 1)\ot 1&=&1\ot(1\ot 1),\\
d_{0}(1\ot 1)&=&d_{0}\ot 1+1\ot d_{0},\\
d_{0}d_{0}&=&0.
\end{eqnarray*}
The operad $\Q$ is a graded operad, $\Q=(\Q^{i})$,
whose degree is defined as the number of $d_{0}$.
In \cite{U1}, it was proved that the operad $\s\Perm$ is isomorphic to the suboperad of $\Q$
whose $n$th component is $\Q^{n-1}(n)$, that is,
\begin{lemma}
$\s\Perm\cong(\Q^{n-1}(n))$.
\end{lemma}
\begin{proof}(Sketch)
We check that $(\Q^{n-1}(n))$ satisfies the relation of $\s\Perm$.
The odd version of associative law is
$$
(d_{0}\ot 1)\c_{1}(d_{0}\ot 1)=-(d_{0}\ot d_{0}\ot 1)=-(d_{0}\ot 1)\c_{2} (d_{0}\ot 1)
$$
and the odd premutation relation is
$$
(d_{0}\ot 1)\c_{1}(d_{0}\ot 1)=-(d_{0}\ot d_{0}\ot 1)=-(d_{0}\ot 1)\c_{1}(1\ot d_{0}).
$$
\end{proof}
From this, we obtain the following expression of $\s\Perm$,
\begin{eqnarray*}
\s\Perm(1)&=&\mathbb{K},\\
\s\Perm(2)&=&<d_{0}\ot 1 \ , \ 1\ot d_{0}>,\\
\s\Perm(3)&=&<d_{0}\ot d_{0} \ot 1 \ , \ d_{0}\ot 1\ot d_{0} \ , \ 1\ot d_{0}\ot d_{0}>,\\
\cdots &\cdots& \cdots.
\end{eqnarray*}
The proposition below is the binary model of the main theorem of this note.
\begin{proposition}[Chapoton \cite{Chap} see also Vallette \cite{Val}]\label{LLP}
$\s\Leib\cong\Lie\ot\s\Perm$.
\end{proposition}
We prove this proposition by using
the method of derived bracket construction.
\begin{proof}
The elements of $\Lie\ot\s\Perm$
are regarded as {\em derived brackets}, for example,
\begin{eqnarray*}
\{d_{0}(1),2\}&\cong&\{1,2\}\ot(d_{0}\ot 1),\\
\{1,d_{0}(2)\}&\cong&\{1,2\}\ot(1\ot d_{0}),\\
\{\{d_{0}(1),d_{0}(2)\},3\}&\cong&\{\{1,2\},3\}\ot(d_{0}\ot d_{0}\ot 1).
\end{eqnarray*}
The derivation of the bracket $d_{0}\{1,2\}$ is well-defined
as a linear combination of the derived brackets,
$$
d_{0}\{1,2\}:=\{d_{0}(1),2\}+\{1,d_{0}(2)\}.
$$
It is easy to see that $\s\Leib(2)\cong(\Lie\ot\s\Perm)(2)$
and that $\Lie\ot\s\Perm$ is generated by $(\Lie\ot\s\Perm)(2)$.
The derived bracket satisfies the odd Leibniz identity,
\begin{eqnarray*}
\{d_{0}(1),\{d_{0}(2),3\}\}&=&\{\{d_{0}(1),d_{0}(2)\},3\}-\{d_{0}(2),\{d_{0}(1),3\}\} \\
&=&-\{d_{0}\{d_{0}(1),2\},3\}-\{d_{0}(2),\{d_{0}(1),3\}\}.
\end{eqnarray*}
Hence there exists an operadic surjection $\psi:\s\Leib\to\Lie\ot\s\Perm$.
$$
\begin{CD}
\T(\s\Leib(2))@=\T\big((\Lie\ot\s\Perm)(2)\big)\\
@VVV @VVV \\
\s\Leib@>{\psi}>>\Lie\ot\s\Perm.
\end{CD}
$$
By a dimension counting, one can prove that this map is isomorphism.
Since $\s\Perm(n)\cong\Q^{n-1}(n)$, $\dim\s\Perm(n)=n$.
It is well-known that $\dim\Lie(n)=(n-1)!$ and $\dim\Leib(n)=n!$.
Hence $\dim\s\Leib(n)=\dim(\Lie\ot\s\Perm)(n)$ for each $n$.
\end{proof}

%%%%%%%%%%%%%%%%%%%%%%%%%%%%%%%%%%%
\subsection{Deformation of $\Perm$}
%%%%%%%%%%%%%%%%%%%%%%%%%%%%%%%%%%%

Let $V:=<d_{1},d_{2},...>$ be a space of 1-ary operators of degree $+1$
and let $1\ot 1\in\Com(2)$ the same as above.
Define a quadratic operad,
$$
\O:=\T(V,1\ot 1)/(R_{\O}),
$$
where $R_{\O}$ is the space of quadratic relations,
\begin{eqnarray*}
(1\ot 1)\ot 1&=&1\ot(1\ot 1),\\
d_{n}(1\ot 1)&=&d_{n}\ot 1+1\ot d_{n}, \ \ \forall n\in\mathbb{N}.
\end{eqnarray*}
We should remark that $\O(1)$ is the same as the tensor algebra over $V$,
\begin{equation}\label{defo1}
\O(1)=\mathbb{K}\oplus V\oplus{V}^{\ot 2}\oplus{V}^{\ot 3}\oplus\cdots.
\end{equation}
We define on the operad $\O$ the second degree which is called a {\em weight}.
The weight function is denoted by $w(-)$.
\begin{definition}[weight on $\O$]
$w(d_{n}):=n$ and $w(1^{\ot n}):=1-n$ for each $n$.
\end{definition}
Then $\O$ becomes a graded and weighted operad $\O=(\O^{g,w})$.
The degree and the weight are both additive with respect to the
operad structure of $\O$.
Hence the sub $\S$-module of weight $0$, $(\O^{\bullet,0})$, becomes
a suboperad of $\O$.
We introduce the main object of this note :
\begin{definition}[deformation operad]
$\D_{\infty}:=(\O^{\bullet,0})$.
\end{definition}
For each $n\in\mathbb{N}$,
$$
\D_{\infty}(n)=\O^{1,0}(n)\oplus\O^{2,0}(n)\oplus\cdots\oplus\O^{n-1,0}(n).
$$
The elementary parts of $\D_{\infty}$ have the following form,
\begin{eqnarray*}
\D_{\infty}(1)&=&\mathbb{K}.\\
\D_{\infty}(2)&=&<d_{1}\ot 1 \ , \ 1\ot d_{1}>.\\
\D_{\infty}^{1}(3)&=&<d_{2}\ot 1\ot 1 \ , \ 1\ot d_{2}\ot 1 \ , \ 1\ot 1\ot d_{2}>.\\
\D_{\infty}^{2}(3)&=&<d_{1}\ot d_{1}\ot 1 \ , \ d_{1}\ot 1\ot d_{1} \ , \ 1\ot d_{1}\ot d_{1} \ , \ \\
& & \ \ \ \ \ d_{1}^{2}\ot 1\ot 1 \ , \ 1\ot d_{1}^{2}\ot 1 \ , \ 1\ot 1\ot d_{1}^{2}>.\\
\cdots&\cdots&\cdots.
\end{eqnarray*}
It is obvious that $\D^{1}_{\infty}(n)\cong s\mathbb{K}^{n}$ for each $n$.\\
\indent
One can easily prove that
\begin{proposition}\label{proofap}
The deformation operad is generated by $\D_{\infty}^{1}$.
\end{proposition}
\begin{proof}(Sketch)
We call a monomial of derivations $d_{i_{1}}d_{i_{2}}\cdots d_{i_{f}}$
a higher order derivation of order $f$.
A homogeneous element $X:=x_{1}\ot\cdots\ot x_{n}\in\D_{\infty}^{a}(n)$
is called a \textbf{higher derived product of order} $a$,
if the derivations in $X$ are all the same position
(e.g. $1\ot d_{1}d_{2}\ot 1\ot 1$.)
Generators of $\D^{1}_{\infty}$ are special higher derived products of order $1$.
It is easy to prove that the deformation operad is generated by higher derived products.
Hence the problem is reduced to proving that
the higher derived product of any order is generated by $\D^{1}_{\infty}$.
This will be solved by using induction w.r.t. degree.
\end{proof}
\indent
Let us define on the deformation operad a differential.
If the derivations $d_{1},d_{2},...$ are deformations of a formal derivation $d_{0}$
(not necessarily differential),
then the deformation derivation $d_{\hbar}:=d_{0}+\hbar d_{1}+\hbar^{2}d_{2}+\cdots$
satisfies $d_{\hbar}d_{\hbar}=d_{0}d_{0}(\neq 0)$, or equivalently,
\begin{equation}\label{dddd}
[d_{0},d_{n}]+\sum_{i+j=n}d_{i}d_{j}=0,
\end{equation}
where $[d_{0},d_{n}]$ is the graded commutator and $i,j,n\ge 1$.
By using (\ref{dddd}) one can define a differential on $\D_{\infty}$.
\begin{definition}[differential on $\D_{\infty}$]\label{defdonD}
For any $x_{1}\ot\cdots\ot x_{n}\in\D_{\infty}(n)$,
\begin{eqnarray*}
\d(x_{1}\ot\cdots\ot x_{n})&:=&\sum_{i=1}^{n}(-1)^{|x_{1}|+\cdots+|x_{i-1}|}
(x_{1}\ot\cdots\ot\d x_{i}\ot\cdots x_{n}),\\
\d(x_{i})&:=&[d_{0},x_{i}].
\end{eqnarray*}
\end{definition}
The homogeneous condition $\d\d=0$ is followed
from the Bianchi identity\footnote{
By definition, $\d(d_{1})=0$, which yields $\d\d(d_{2})=0$,
which yields $\d\d(d_{3})=0$,... forever.
}.
Therefore, $(\D_{\infty},\d)$ becomes a dg operad.
For example,
\begin{eqnarray*}
\d(d_{2}\ot 1\ot 1)&=&-d_{1}^{2}\ot 1\ot 1\\
&=&-(d_{1}\ot 1)\c_{1}(d_{1}\ot 1)-(d_{1}\ot 1)\c_{2}(d_{1}\ot 1),
\end{eqnarray*}
which yields
\begin{lemma}\label{hdperm}
$\H^{top}(\D_{\infty},\d)\cong\s\Perm$.
\end{lemma}
\begin{remark}
$\D_{\infty}\neq\s\Perm_{\infty}$.
\end{remark}
%%%%%%%%%%%%%%%%%%%%%%%%%%%%%%%%%%%%%%%%
\subsection{Dimension of $\D_{\infty}$}
%%%%%%%%%%%%%%%%%%%%%%%%%%%%%%%%%%%%%%%%

In this section we compute the dimension of the deformation operad.
Consider a set of derivations,
$A=\{d_{1}^{(\lambda_{1})} \ , \ d_{2}^{(\lambda_{2})} \ , \ ... \ , \ d_{n-1}^{(\lambda_{n-1})}\}$,
where $d_{i}^{(\lambda_{i})}$ is the $\lambda_{i}$-copies of $d_{i}$
(e.g. $d_{i}^{(3)}=\{d_{i},d_{i},d_{i}\}$ and $d_{i}^{(0)}:=\emptyset$.)
We define a space $\Delta^{(\lambda_{1},\lambda_{2},...,\lambda_{n-1})}(n)$
as a subspace of $\D_{\infty}^{a}(n)$ such that each of elements
has all derivations in $A$.
Here the degree $a$ is equal to the cardinal number of $A$.
For example, when $n=5$ and $a=2$,
\begin{eqnarray*}
\Delta^{(0,2,0,0)}(5)&=&<d_{2}d_{2}\ot 1\ot 1\ot 1 \ , \ d_{2}\ot d_{2}\ot 1\ot 1 \ ,\ ...>,\\
\Delta^{(1,0,1,0)}(5)&=&<d_{1}d_{3}\ot 1\ot 1\ot 1 \ , \ d_{1}\ot d_{3}\ot 1\ot 1 \ ,\ ...>,
\end{eqnarray*}
which are subspaces of $\D^{2}_{\infty}(5)$.
In particular, the top degree part of $\D_{\infty}(n)$ is
$$
\D^{top}_{\infty}(n)=\D^{n-1}_{\infty}(n)=\Delta^{(n-1,0,...,0)}(n).
$$
From the assumptions of degree and weight, we obtain
two natural conditions in combinatorial theory,
\begin{eqnarray}
\label{dc01}\lambda_{1}+\cdots+\lambda_{n-1}&=&a,\\
\label{dc02}\lambda_{1}+2\lambda_{2}+\cdots+(n-1)\lambda_{n-1}&=&n-1.
\end{eqnarray}
The dimension of $\Delta^{(\lambda_{1},\lambda_{2},...,\lambda_{n-1})}(n)$
is easily computed :
\begin{lemma}
\begin{equation}\label{dimdelta}
\dim\Delta^{(\lambda_{1},\lambda_{2},...,\lambda_{n-1})}(n)=
\binom{n+a-1}{a}\frac{a !}{\lambda_{1}!\cdots\lambda_{n-1}!}.
\end{equation}
\end{lemma}
\begin{proof}
This is a kind of balls and boxes questions, i.e.,
$1^{\ot n}$ is $n$-boxes and $d_{1},d_{2},...$ are balls.
\end{proof}
Since $\Delta^{(...)}(n)$ is a direct summand of $\D_{\infty}^{a}(n)$, we obtain
\begin{equation}\label{koshiki}
\dim\D_{\infty}^{a}(n)=
\sum_{(\ref{dc01}),(\ref{dc02})}\binom{n+a-1}{a}\frac{a !}{\lambda_{1}!\cdots\lambda_{n-1}!}=
\binom{n+a-1}{a}\binom{n-2}{a-1},
\end{equation}
which yields
\begin{proposition}\label{propdim}
$$
\dim\D_{\infty}(n)=\sum_{a=1}^{n-1}\binom{n+a-1}{a}\binom{n-2}{a-1}.
$$
When $n=1$, $\dim\D_{\infty}(1)=1$ because $\D_{\infty}(1)=\mathbb{K}$.
\end{proposition}
In (\ref{koshiki}), we used an well-known formula.

%%%%%%%%%%%%%%%%%%%%%%%%%%%%%%%%%%
\section{Higher derived brackets}
%%%%%%%%%%%%%%%%%%%%%%%%%%%%%%%%%%

We study the operad $\Lie\ot\D_{\infty}$ with a differential $\Lie\ot\d$.
We denote by $\{1,2\}$ the Lie bracket in $\Lie(2)$
and by $\{1,2,...,n\}=\{\{\{1,2\},...\},n\}$ the left $n$-fold bracket in $\Lie(n)$.
The elements of $\Lie\ot\D_{\infty}$ are identified to Lie brackets
whose leaves are derived (recall the proof of Proposition \ref{LLP}.)
The brackets in $\Lie\ot\D^{1}_{\infty}$ are called
the \textbf{higher derived brackets}
and the higher derived bracket is said to be \textbf{normal},
if the Lie bracket is left-normed and if the derivation acts
on the leaf of the most left-side, that is,
$$
\{d_{n}(l_{1}),l_{2},...,l_{n+1}\},
$$
where $l_{1},l_{2},...,l_{n+1}$ are labels.
The set of the normal higher derived brackets forms a linear base of $\Lie\ot\D^{1}_{\infty}$.
\begin{proposition}
$(\Lie\ot\D^{1}_{\infty})(n)\cong sS_{n}$ for each $n\ge 2$.
\end{proposition}
\indent
We recall a classical lemma for free Lie algebra.
\begin{lemma}[Elimination Theorem \cite{NB}]
Let $\Delta\sqcup\mathbb{N}$ be a word set
and let $\F_{Lie}(\Delta\sqcup\mathbb{N})$ be the free Lie algebra over the set,
where $\Delta:=<\delta_{1},\delta_{2},...>$
and where the degree of $\delta_{n}$ is $+1$ for each $n$.
Then
$$
\F_{Lie}(\Delta\sqcup\mathbb{N})\cong\F_{Lie}(T)\oplus\F_{Lie}(\mathbb{N}),
$$
where
$T:=\Delta\oplus\{\Delta,\mathbb{N}\}\oplus\{\Delta,\mathbb{N},\mathbb{N}\}\oplus\cdots$.
\end{lemma}
\begin{proof}
See A1 in Appendix.
\end{proof}
\begin{proposition}
$\Lie\ot\D_{\infty}$ is generated by the higher derived brackets of normal.
\end{proposition}
\begin{proof}
Since the rule of derivation is the same as the Jacobi identity,
the operad $\Lie\ot\D_{\infty}$ can be embedded linearly in
the free Lie algebra $\F_{Lie}(\Delta\sqcup\mathbb{N})$,
via the adjoint representation $d_{n}(-)\cong\{\delta_{n},-\}$.
For instance,
\begin{equation}\label{theq2}
\{d_{2}d_{1}1,2,3,4\}\cong\{\delta_{2},\{\delta_{1},1\},2,3,4\}.
\end{equation}
By the elimination theorem,
the target of this embedding, $\phi$, is $\F_{\Lie}(T)$,
$$
\begin{CD}
\phi:\Lie\ot\D_{\infty}@>{\subset}>>\F_{\Lie}(\Delta\sqcup\mathbb{N})@>{proj.}>>\F_{\Lie}(T).
\end{CD}
$$
Thus an arbitrary monomial $\mu\in\Lie\ot\D_{\infty}$
is expressed as a polynomial in $\F_{\Lie}(T)$,
\begin{equation}\label{decm}
\phi:\mu\mapsto\sum\{t_{1},...,t_{a}\},
\end{equation}
where $a$ is the degree of $\mu$
and $t_{i}=\{\delta_{j},l_{1},...,l_{f}\}$
is a homogeneous element of $T$.
In the following, we identify $\mu\cong\phi(\mu)$.
\begin{definition}[weight on $\F_{\Lie}(T)$]
$w(\delta_{n}):=n+1$ and $w\{\cdot,\cdot\}:=-1$.
\end{definition}
We have $w(d_{n})=w\{\delta_{n},-\}=n$, namely, $\phi$ preserves the weight.
Hence the weight of the monomial $\{t_{1},...,t_{a}\}$ which arises in (\ref{decm}) is zero.
From this, we notice that in $\{t_{1},...,t_{a}\}$
there exists $t_{i}$ whose weight is non positive.
Such a $t_{i}$ has the form of
\begin{equation}\label{theq3}
t_{i}=\{\delta_{j},l_{1},...,l_{f}\}\cong\{\{d_{j}(l_{1}),...,l_{j+1}\},...,l_{f}\},
\end{equation}
where $f\ge j+1$.
Without loss of generality, one can put $\mu:=\{t_{1},...,t_{a}\}$.
From (\ref{theq3}),
we obtain a natural decomposition of $\mu$,
\begin{equation}\label{decmu}
\mu=\nu(x)\c_{x}\{d_{j}(l_{1}),...,l_{j+1}\}.
\end{equation}
By the assumption of induction w.r.t. the degree,
$\nu(x)$ is generated by higher derived brackets.
Therefore, $\mu$ is also so.
\end{proof}
From Proposition above, we obtain an operadic epi-morphism,
$$
\theta:\T(\Lie\ot\D^{1}_{\infty})\to\Lie\ot\D_{\infty}.
$$
Here $\theta$ is homogeneous.
We prove that $\theta$ is mono by using the method of dimension counting.
\begin{lemma}\label{keylemma}
$\T(\Lie\ot\D^{1}_{\infty})\cong\Lie\ot\D_{\infty}$.
\end{lemma}
\begin{proof}
From Proposition \ref{propdim} we obtain
\begin{equation}\label{dim01}
\dim(\Lie\ot\D_{\infty})(n)
=(n-1)!\sum_{a=1}^{n-1}\binom{n+a-1}{a}\binom{n-2}{a-1},
\end{equation}
where an well-known condition $\dim\Lie(n)=(n-1)!$ is used.\\
\indent
Since $(\Lie\ot\D^{1}_{\infty})(n)\cong sS_{n}$,
the free operad $\T(\Lie\ot\D^{1}_{\infty})$ is
just the labeled planar rooted trees (see Fig 1).
The cardinal number of non-labeled planar rooted trees
is known as \textbf{Schr\"{o}der number} (see Table 1.)
Hence the dimension of $\T(\Lie\ot\D^{1}_{\infty})(n)$ is
$$
\dim\T(\Lie\ot\D^{1}_{\infty})(n)=n!\cdot s(n),
$$
where
$s(n)$ is the Schr\"{o}der number for $n$-trees
and $n!$ is the cardinal of labels.
By using a result in Gessel \cite{Gess} (see also Rogers \cite{Rog}), one can prove that
$$
s(n)=\frac{1}{n!}\dim(\Lie\ot\D_{\infty})(n).
$$
Therefore, for each $n$ $\dim\T(\Lie\ot\D^{1}_{\infty})(n)=\dim(\Lie\ot\D_{\infty})(n)$,
which implies that $\theta$ is an operadic isomorphism.
\end{proof}
%\begin{remark}[Schr\"{o}der bracketing and higher derived bracketing]
The cardinal number of all good brackets is also the Schr\"{o}der number
(so-called Schr\"{o}der bracketing.)
For instance, $((xx)(xxxx)x)$ this is a Schr\"{o}der bracketing of arity $7$
consisting of $(xx)$, $(xxxx)$ and $(xxx)$.
The lemma above says that
the higher derived bracketing is equivalent to
the labeled-Schr\"{o}der bracketing.
\medskip\\
\indent
%\end{remark}
Now we give the main result of this note.
\begin{theorem}
$\s\Leib_{\infty}\cong\Lie\ot\D_{\infty}$ as a dg-operad.
\end{theorem}
\begin{proof}
From Lemma \ref{firstlemma},
$$
\s\Leib_{\infty}\cong\T(\Lie\ot\D^{1}_{\infty})\cong\Lie\ot\D_{\infty}.
$$
The differential $\Lie\ot\d$ on $\Lie\ot\D_{\infty}$
is the same as the tree-differential on $\s\Leib_{\infty}$.
This claim is followed from Proposition \ref{proplemma}
(See A2 in Appendix for a direct proof.)
For example,
\begin{eqnarray*}
(\Lie\ot\d)\{d_{2}(1),2,3\}&=&-\{d_{1}^{2}(1),2,3\} \\
&=&-\{d_{1}\{d_{1}(1),2\},3\}-\{\{d_{1}(1),d_{1}(2)\},3\} \\
&=&-\{d_{1}\{d_{1}(1),2\},3\}-\{d_{1}(1),\{d_{1}(2),3\}\}-\{d_{1}(2),\{d_{1}(1),3\}\} \\
&=&-[[1,2],3]-[1,[2,3]]-[2,[1,3]],
\end{eqnarray*}
where $[1,2]:=\{d_{1}(1),2\}$.
\end{proof}
Since $\Leib$ is Koszul (cf. Loday \cite{Lod1,Lod2}), we obtain
\begin{corollary}
The dg operad $(\D_{\infty},\d)$ is a resolution over $\s\Perm$.
\end{corollary}
\begin{proof}
By Lemma \ref{hdperm}.
\end{proof}
%It is known that if the homotopies of sh Leibniz algebra are all skewsymmetric,
%then it is an sh Lie algebra. This claim holds on the level of operad.
%\begin{corollary}
%$\s\Lie_{\infty}\cong\Lie\ot\D_{\infty}/(\{.,.\})$,
%where $(\{.,.\})$ is an ideal generated by the brackets which
%include pure Lie brackets of non derived.
%\end{corollary}
%\begin{proof}(Sketch)
%The higher derived brackets are {\em odd}
%skewsymmetric modulo the pure Lie bracket,
%for instance,
%\begin{eqnarray*}
%\{d_{2}(1),2,3\}&=&\{d_{2}(1),3,2\}+\{d_{2}(1),\{2,3\}\},\\
%\{d_{2}(1),2,3\}&=&\{d_{2}\{1,2\},3\}+\{d_{2}(2),1,3\},
%\end{eqnarray*}
%where $\{2,3\}$ and $\{1,2\}$ are pure Lie brackets.
%That is, $\{d_{2}(1),\{2,3\}\}$ and $\{d_{2}\{1,2\},3\}$ are elements of the ideal.
%\end{proof}
\begin{remark}[On sh associative operad]
As an operad $\s\Leib_{\infty}$ is isomorphic to $\s\Ass_{\infty}$.
Hence it is natural to ask how much different
the tree-differential on $\s\Ass_{\infty}$ is from $\Lie\ot\d$.
It is easy to answer this question.
The differential $\Lie\ot\d$ is decomposed into
regular part and non regular one.
Here the word ``regular" means that $\sigma=id$ in (\ref{TT03}).
For example, in
$$
(\Lie\ot\d)\{d_{2}(1),2,3\}=-[[1,2],3]-[1,[2,3]]-[2,[1,3]],
$$
$(\Lie\ot\d)\{d_{2}(1),2,3\}=-[[1,2],3]-[1,[2,3]]$ is regular,
and $-[2,[1,3]]$ is nonregular.
The regular part of $\Lie\ot\d$
is just the tree-differential on $\s\Ass_{\infty}$.
\end{remark}

In the final of this note,
we study a problem of counting the number of trees.
An $n$-corolla, which is denoted by $c_{n}$,
is a non-labeled planar rooted tree with $n$-leaves,
$1$-root and $1$-internal vertex (see Fig 3.)
An arbitrary tree is generated from corollas by grafting of trees.
\begin{figure}[h]
\begin{center}
\includegraphics[width=70mm]{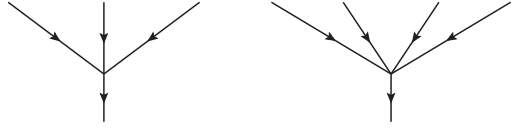}
\end{center}
\caption{$c_{3}$ and $c_{4}$}
\end{figure}
Let $C:=\{c_{2}^{(\lambda_{1})} \ , \ c_{3}^{(\lambda_{2})} \ , \ ... \ , \ c_{n}^{(\lambda_{n-1})}\}$
be a set of corollas, where $c_{i+1}^{(\lambda_{i})}$ is the $\lambda_{i}$-copies of $c_{i+1}$
(like $d_{i}^{(\lambda_{i})}$ in Section 2.)
Let $TC$ be the set of trees generated by $C$.
For example, if $C=\{c_{2},c_{3}\}$,
$$
T\{c_{2},c_{3}\}=\{c_{2}\c_{1}c_{3} \ , \ c_{2}\c_{2}c_{3} \ ,
\ c_{3}\c_{1}c_{2} \ , \ c_{3}\c_{2}c_{2} \ , \ c_{3}\c_{3}c_{2}\}.
$$
where $\c_{i}$ is the grafting product of trees at the $i$th-leaf.
Hence the cardinal number of $T\{c_{2},c_{3}\}$ is $5$.
The number of leaves of $T\in TC$ is computed as follows.
$$
|T|:=\lambda_{1}+2\lambda_{2}+\cdots+(n-1)\lambda_{n-1}+1.
$$
\begin{corollary}
The cardinal number of $TC$ is
$$
\card(TC):=\frac{1}{|T|}\binom{|T|+\Lambda-1}{\Lambda}\frac{\Lambda !}{\lambda_{1}!\cdots\lambda_{n-1}!},
$$
where $\Lambda=\lambda_{1}+\cdots+\lambda_{n-1}$.
\end{corollary}
\begin{proof}
$\lambda_{i}$ is equal to the number of $d_{i}$ and $\Lambda=a$.
\end{proof}
It is easy to see that $\card(T\{c_{k}^{n}\})$
is the (Fuss-)Catalan number for $k$-ary trees.
\begin{center}
\textbf{-- Appendix --}
\end{center}
\noindent
\textbf{A1} (\cite{NB}).
Let $X:=S^{c}\sqcup S$ be a wordset decomposed
into a subset $S$ and its complement $S^{c}$,
and let $\F_{Lie}(X)$ be the free Lie algebra over $X$.
Then the following identity holds.
$$
\F_{Lie}(X)\cong\F_{Lie}(T)\oplus\F_{Lie}(S),
$$
where $T$ is a word set,
$$
T:=S^{c}\sqcup(S^{c},S)\sqcup(S^{c},S,S)\sqcup\cdots.
$$
And there exists a natural isomorphism,
$$
T\cong S^{c}\oplus\{S^{c},S\}\oplus\{\{S^{c},S\},S\}\oplus\cdots.
$$
\begin{proof}
$$
\F_{Lie}(S^{c}\sqcup S)\cong
\frac{\F_{Lie}(T\sqcup S)}{I}\cong\F_{Lie}(T)\oplus\F_{Lie}(S),
$$
where $I$ is an ideal generated by the identity,
$\{T,S\}-(T,S)=0$.
\end{proof}
\noindent
\textbf{A2}.
$\Lie\ot\d$ is the tree differntial on $\s\Leib_{\infty}$.
\begin{proof}
Recall (\ref{TT01})-(\ref{TT03})
the defining equations of the tree differential on $\s\Leib_{\infty}$.
It suffices to prove that $\{[d_{i-1},d_{j-1}](1),...,n\}\cong T_{i}T_{j}+T_{j}T_{i}$.
The left-hand side expands to
$$
\{[d_{i-1},d_{j-1}](1),...,n\}=\{d_{i-1}d_{j-1}(1),...,n\}+\{d_{j-1}d_{i-1}(1),...,n\}
$$
and the term is
\begin{equation}\label{proofa201}
\{d_{i-1}d_{j-1}(1),...,n\}=
\{d_{i-1}\{d_{j-1}(1),...,j\},...,n\}+\sum_{k\ge 2}^{j}\{d_{j-1}(1),...,d_{i-1}(k),...,n\},
\end{equation}
where the derivation property is used.
We denote by $T^{(m)}$ a labeled rooted tree whose most left label is $m$.
One can divide $T_{i}T_{j}+T_{j}T_{i}$ into two parts, i.e.,
the part that $T_{j}$ has the lable $1$ and the other part,
$$
T_{i}T_{j}+T_{j}T_{i}=\Big(T_{i}T_{j}^{(1)}+T_{j}^{(1)}T_{i}\Big)
+\Big(T_{i}^{(1)}T_{j}+T_{j}T_{i}^{(1)}\Big).
$$
The first term has the form of
\begin{equation}\label{proofa202}
T_{i}T_{j}^{(1)}+T_{j}^{(1)}T_{i}=T_{i}\c_{1}T_{j}+
\sum_{k\ge 2}^{j}T_{i}^{(k)}T_{j}^{(1)}+T_{j}^{(1)}T_{i}^{(k)}.
\end{equation}
Obviously $\{d_{i-1}\{d_{j-1}(1),...,j\},...,n\}\cong T_{i}\c_{1}T_{j}$.
It is also easy to see that
$$
\{d_{j-1}(1),...,d_{i-1}(k),...,n\}\cong T_{i}^{(k)}T_{j}^{(1)}+T_{j}^{(1)}T_{i}^{(k)},
$$
which yields $(\ref{proofa201})\cong(\ref{proofa202})$.
\end{proof}

Keywords : derived bracket, Leibniz Loday algebra, operad, homotopy algebra,
tree graph, Schr\"{o}der number.
\begin{verbatim}
Kyousuke UCHINO (Freelance)
SAWAYA apartments 103, Yaraicho 9
Shinjyuku Tokyo Japan
email:kuchinon[at]gmail.com
\end{verbatim}
\end{document}